\documentclass[letterpaper, 10pt, conference]{ieeeconf}  %

\IEEEoverridecommandlockouts                              %

\usepackage{graphics} %
\usepackage{epsfig} %
\usepackage{amsmath} %
\usepackage{amssymb}  %

\usepackage{hyperref}
\usepackage{cleveref}
\usepackage{xcolor}
\usepackage{subfigure}
\usepackage{widetext}
\usepackage{xspace}

\usepackage{amsthm}
\usepackage[symbol]{footmisc}

\newcommand{\ellone}{$\mathcal{L}_1$\xspace}

\DeclareMathOperator*{\argmax}{arg\,max}
\DeclareMathOperator*{\argmin}{arg\,min}

\newtheorem{theorem}{Theorem}
\newtheorem{corollary}{Corollary}
\newtheorem{lemma}{Lemma}

\newtheorem{definition}{Definition}
\newtheorem{assumption}{Assumption}
\newtheorem{remark}{Remark}

\newif\ifisnotpreprint
\isnotpreprintfalse  %

\allowdisplaybreaks

\title{\LARGE \bf
\ellone Adaptive Optimizer for Online Time-Varying Convex Optimization
}

\author{Jinrae Kim$^{1}$ and Naira Hovakimyan$^{1}$%
\thanks{*This work is financially supported by
Air Force Office of Scientific Research (AFOSR) grant \#FA9550-21-1-0411,
National Aeronautics and Space Administration (NASA) grant \#80NSSC20M0229,
NASA University Leadership Initiative (ULI) grant \#80NSSC22M0070,
and National Science Foundation (NSF) Information and Intelligent Systems (IIS) grant \#23-31878.
}%
\thanks{$^{1}$Jinrae Kim and Naira Hovakimyan are with Department of Mechanical Engineering,
The Grainger College of Engineering,
University of Illinois Urbana-Champaign, Urbana, IL 61801
        {\tt\small \{jinrae,nhovakim\}@illinois.edu}}%
}

\begin{document}

\maketitle
\pagestyle{plain}
\pagenumbering{arabic}

\begin{abstract}
We propose an adaptive method for online time-varying (TV) convex optimization, termed $\mathcal{L}_{1}$ adaptive optimization ($\mathcal{L}_{1}$-AO).
TV optimizers utilize a prediction model to exploit the temporal structure of TV problems,
which can be inaccurate in the online implementation.
Inspired by $\mathcal{L}_{1}$ adaptive control,
the proposed method augments an adaptive update law to estimate and compensate for the uncertainty from the prediction inaccuracies.
The proposed method provides performance bounds of the error in the optimization variables and cost function,
allowing efficient and reliable optimization for TV problems.
\ifisnotpreprint
\else
Numerical simulation
\fi
results demonstrate the effectiveness of the proposed method for online TV convex optimization.
\end{abstract}

\section{INTRODUCTION}
Optimization is critical across a wide range of applications.
As streaming data are available in modern engineering systems,
it is necessary to solve the optimization problem in real-time in dynamic environments~\cite{camachoModelPredictiveControl2007,arslanExactRobotNavigation2016,arslan2019sensor,yang2016online}.
The challenge is that quite often the optimization objective and constraint functions change over time.
These problems are referred to as \textit{time-varying} (TV) optimization.

Solving TV optimization problems at each time instant by conventional time-invariant methods may be computationally burdensome.
Instead,
one can exploit the underlying temporal structure (dynamics)
to predict how the problem varies over time and update the optimization variables accordingly,
reducing the computational load significantly~\cite{simonettoTimeVaryingConvexOptimization2020a}.
\begin{figure}[t!]
    \centering
    \subfigure[]{\includegraphics[width=0.45\textwidth]{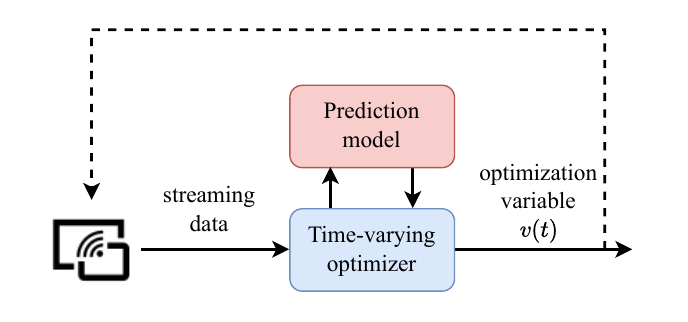}}
    \subfigure[]{\includegraphics[width=0.45\textwidth]{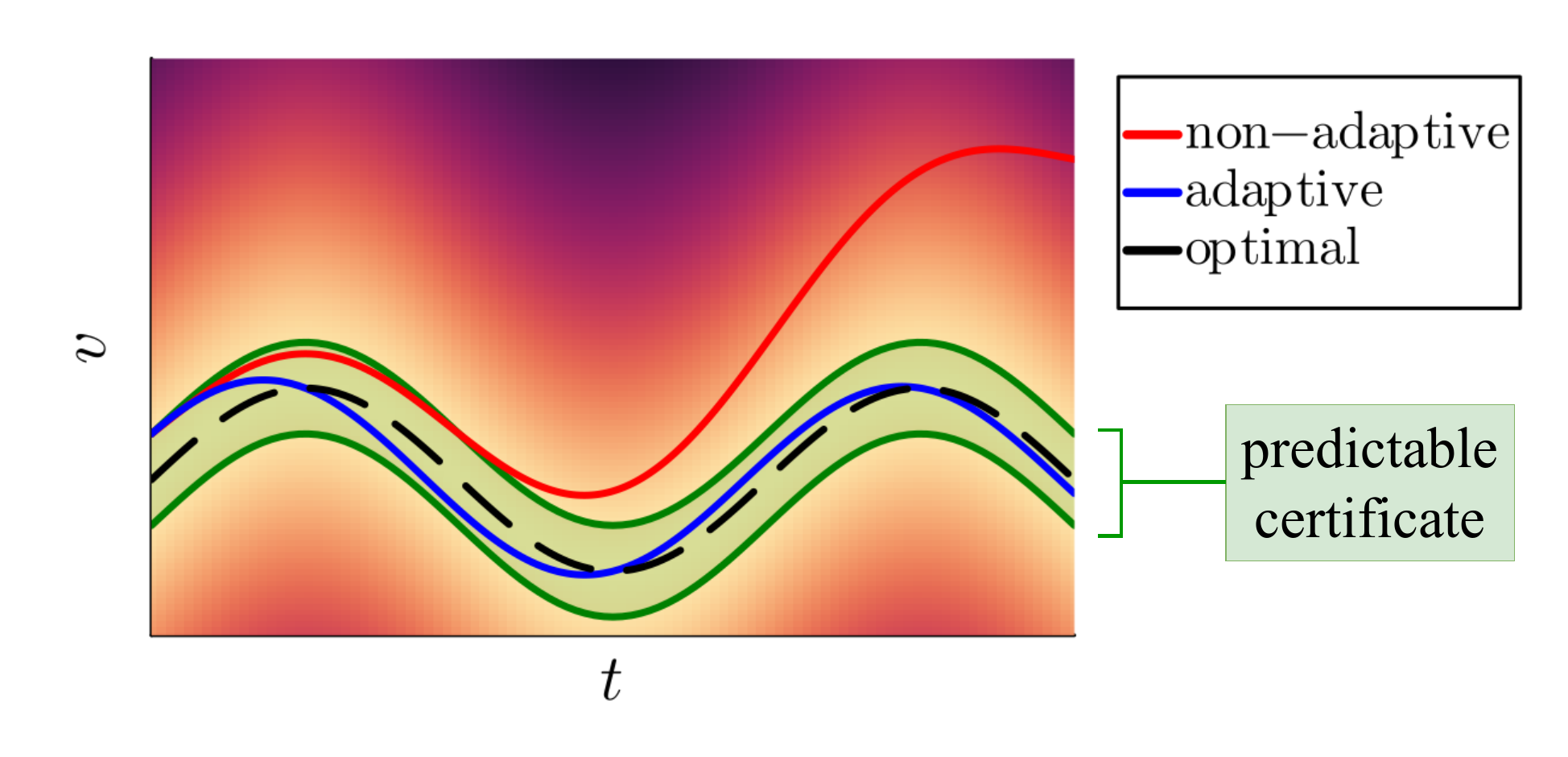}}
    \caption{
    (a)
    Time-varying (TV) optimization methods
    update the optimization variable continuously using a prediction model to track the optimal solution.
    (b)
    In the online implementation,
    the prediction model may be inaccurate, leading to performance degradation.
    The proposed $\mathcal{L}_{1}$ adaptive optimization ($\mathcal{L}_{1}$-AO) compensates for the uncertainty from prediction inaccuracies and provides predictable robustness certification for TV convex optimization problems.
    The heatmap indicates the TV cost function.
    }
    \label{fig:tvopt}
\end{figure}
In practice,
it is challenging to accurately predict how the cost and constraint functions change over time.
The prediction model is employed based on imperfect prior knowledge,
making the underlying dynamics \textit{uncertain} in online implementation.
The inaccuracy of the prediction model may degrade the performance of the TV algorithm significantly.
Therefore, it is crucial to compensate for the uncertainty from the prediction inaccuracies for efficient and reliable TV optimization.

\textbf{Statement of Contributions:}
We propose an adaptive method for continuous-time (CT) online TV convex optimization,
which recovers the performance of the baseline optimizer by compensating for the uncertainty from the prediction inaccuracies. This allows \textit{robustness certification} w.r.t. the error in optimization variables and optimality gap.
\autoref{fig:tvopt} illustrates the proposed method.
Let $\Phi(t, v)$ be the smooth and strongly convex cost function
such that $\nabla_{vv}\Phi \succeq m_{f} I$ for some $m_{f} > 0$.
The optimization variable at time $t$ is denoted by $v(t)$.
The optimal solution $v^{\star}$ attains zero gradient, i.e.,
$v^{\star}(t) = \argmin_{v} \Phi(t, v) \iff \nabla_{v} \Phi(t, v^{\star}(t)) = 0$.
The gradient dynamics are given as
$\dot{\nabla}_{v} \Phi = \nabla_{vt}\Phi + \nabla_{vv} \Phi \dot{v}$.
Consider a baseline TV optimizer $\dot{v}_{b}$, designed based on the approximate gradient dynamics $\dot{\nabla}_{v} \Phi = \hat{\nabla}_{vt}\Phi + \nabla_{vv} \Phi \dot{v}_{b}$ with a prediction model $\hat{\nabla}_{vt} \Phi$, i.e., $\lVert \nabla_{v}\Phi(t, v(t)) \rVert \to 0$ and $\lVert v(t) - v^{\star}(t) \rVert \to 0$ as $t \to \infty$ under the approximate dynamics with $\dot{v} = \dot{v}_{b}$.

\noindent\underline{\textit{Performance}}:
We define the performance by the gradient of the cost function and the tracking error of the optimization variable.
First,
the proposed method provides a feedback update law $\dot{v} = \dot{v}_{b} + \dot{v}_{a}$
such that the gradient is uniformly bounded by $\rho > 0$, i.e.,
$\lVert \nabla_{v} \Phi(t, v(t)) \rVert \leq \rho$, $\forall t \in [0, t_{f}]$,
where $\dot{v}_{a}$ is the adaptive update term,
$t_{f} \leq \infty$ is the operation horizon,
and $\rho > 0$ is a positive number that is computable \textit{a priori}.
Second, the proposed method ensures that
the tracking error of the optimization variable is uniformly bounded as
$\lVert v(t) - v^{\star}(t) \rVert \leq \rho / m_{f}, \forall t \in [0, t_{f}]$.
It implies that the proposed method can provide performance bounds in the presence of prediction inaccuracies.

\noindent\underline{\textit{Optimality}}:
We define the optimality by the gap of the cost function, i.e., $\Phi(t, v(t)) - \Phi(t, v^{\star}(t))$.
Similar to the performance,
the proposed method guarantees the uniform boundedness of the optimality gap as
$0 \leq \Phi(t, v(t)) - \Phi(t, v^{\star}(t)) \leq \rho^{2} / m_{f}, \forall t \in [0, t_{f}]$.
It also implies that the proposed method can provide the optimality bound in the presence of prediction inaccuracies.

\textbf{Related work:}
TV optimization algorithms have been developed in different time settings including discrete-time (DT)~\cite{simonettoPredictionCorrectionAlgorithmsTimeVarying2017} and CT systems~\cite{fazlyabPredictionCorrectionInteriorPointMethod2018,rahiliDistributedContinuousTimeConvex2017}.
CT algorithms can be considered as the limit case of DT algorithms, desirable for real-time applications such as robotics and control~\cite{simonettoTimeVaryingConvexOptimization2020a,kim_smooth_2025}.
To improve the tracking performance,
a TV optimization algorithm can exploit a model to predict how the TV problem evolves.
If such a prediction is not considered, the solution may drift away over time~\cite{popkovGradientMethodsNonstationary2005}.
The prediction model can be data-driven or model-based,
and an inaccurate prediction model may lead to performance degradation.
Interested readers are referred to a survey paper~\cite{simonettoTimeVaryingConvexOptimization2020a}
for more details on TV optimization.

\textbf{Notation:}
Let $\mathbb{N}$ and $\mathbb{R}$ be the sets of natural numbers and real numbers, respectively,
and $\mathbb{N}_{0} := \{0\} \cup \mathbb{N}$.
The sets of non-negative and positive real numbers are denoted by $\mathbb{R}_{\geq 0}$ and $\mathbb{R}_{ > 0}$, respectively.
For a natural number $p \in \mathbb{N}$, let $[p] := \{1, \ldots, p\}$.
Given a matrix $A$, let $\underline{\lambda}(A)$ be the smallest real part of $A$'s eigenvalues.
Given $n \in \mathbb{N}$,
$I_{n}$ denotes the $n \times n$ identity matrix.
The function argument may be simplified or omitted if the context is clear, e.g., $\Phi(t) = \Phi(t, v(t))$.

\section{PRELIMINARIES}
\label{sec:preliminaries}
Consider the following TV optimization problem given by
\begin{equation}
\label{eq:tvopt}
    \min_{v} f_{0}(t, v) \quad
    \text{subject to} \ f_{i}(t, v) \leq 0, \quad \forall i \in [p],
\end{equation}
where $t \in \mathbb{R}_{\geq 0}$ denotes the time,
$v \in \mathbb{R}^{n_{v}}$ is the optimization variable,
$f_{0}:  \mathbb{R}_{\geq 0} \times \mathbb{R}^{n_{v}} \to \mathbb{R}$ is the cost function,
and $f_{i}: \mathbb{R}_{\geq 0} \times \mathbb{R}^{n_{v}} \to \mathbb{R}$ are inequality constraints, $\forall i \in [p]$.
Suppose that $f_{i}(t, v)$ are twice differentiable w.r.t. $t$ and three-times differentiable w.r.t. $v$, $\forall i \in \{0\} \cup [p]$,
and $f_{0}(t, \cdot)$ is uniformly strongly convex, i.e., $ \nabla_{vv} f_{0}(t, v) \succeq m_{f} I_{n_{v}}, \forall (t, v)$ for some $m_{f} > 0$.
The problem can be approximated by an unconstrained problem with log-barrier functions as follows~\cite{fazlyabPredictionCorrectionInteriorPointMethod2018}:
\begin{equation}
\label{eq:log_barrier}
    \min_{v} \ \Phi(t, v)
    := f_{0}(t, v)
    - \frac{1}{c(t)} \sum_{i \in [p]}  \log (-f_{i}(t, v)),
\end{equation}
where $c: \mathbb{R}_{\geq 0} \to \mathbb{R}_{> 0}$ is the penalty parameter that controls the quality of the approximation.
In this regard, only unconstrained problems will be considered in this study.
The minimizer can be characterized as $v^{\star}(t) = \argmin_{v} \Phi(t, v)  \iff \nabla_{v} \Phi (t, v^{\star}(t)) = 0$.

Solving the optimization problem in \eqref{eq:log_barrier} for each $t$
may be computationally burdensome in real-time applications.
Instead,
one can exploit the temporal structure to find the minimizer efficiently~\cite{simonettoTimeVaryingConvexOptimization2020a}.
Consider the gradient dynamics as follows:
\begin{equation}
\label{eq:grad_dyn}
    \dot{\nabla}_{v} \Phi(t)
    = \nabla_{vt} \Phi(t) + \nabla_{vv} \Phi(t) \dot{v}(t),
\end{equation}
with $\nabla_{v}\Phi(0) = \nabla_{v}\Phi(0, v(0))$.
The rate of optimization variables, $\dot{v}$, can be viewed as the \textit{control input} of the gradient dynamics.
Then, an update law $\dot{v}(t)$ can be designed to regulate the gradient to zero.

In the online implementation,
the knowledge about $\frac{\partial f_{i}}{\partial t}$ and $\nabla_{vt} f_{i}$, $\forall i \in \{0\} \cup [p]$ may be inaccurate,
which appears in the term $\nabla_{vt} \Phi$ in \eqref{eq:grad_dyn}.
For example, external streaming data $d(t)$, such as system state or environmental data, may determine the optimization problem, i.e., $\Phi(t, v) = \chi(d(t), v)$ for some function $\chi$,
while the knowledge about the dynamics of streaming data, $\dot{d}(t)$, is limited.
Suppose a nominal prediction model $\hat{\nabla}_{vt} \Phi$ is available.
By rearranging \eqref{eq:grad_dyn},
the prediction error $\tilde{\nabla}_{vt} \Phi := \hat{\nabla}_{vt} \Phi - \nabla_{vt} \Phi$ can be viewed as \emph{uncertainty} of the gradient system as
\begin{equation}
\label{eq:grad_dyn_est}
\begin{split}
    \dot{\nabla}_{v} \Phi(t)
    &= \hat{\nabla}_{vt} \Phi(t)
    + \nabla_{vv} \Phi(t) \dot{v}(t)
    - \tilde{\nabla}_{vt} \Phi(t)
    \\
    &=: \hat{\nabla}_{vt} \Phi(t)
    + \nabla_{vv} \Phi(t) (\dot{v}(t) + \sigma(t)),
\end{split}
\end{equation}
where $\sigma(t) := -\nabla_{vv} \Phi^{-1}(t) \tilde{\nabla}_{vt} \Phi(t)$.

We define norm balls and tubes w.r.t. the gradient and the error in optimization variables for the problem statement.
\begin{definition}
     Given $\rho > 0$,
     define the norm balls w.r.t. the gradient and the error in optimization variables as $\mathcal{O}_{\nabla}(t, \rho) := \{v \vert \lVert \nabla_{v} \Phi(t, v) \rVert \leq \rho \}$
     and
     $\mathcal{O}(t, \rho) := \{v \vert \lVert v - v^{\star}(t) \rVert \leq \rho \}$,
     respectively,
     for all $t \geq 0$.
     Also,
     define the tubes w.r.t. the gradient and optimization variables as
     $\mathcal{T}_{\nabla}(\rho) := \cup_{t \geq 0} \mathcal{O}_{\nabla}(t, \rho)$
     and
     $\mathcal{T}(\rho) := \cup_{t \geq 0} \mathcal{O}(t, \rho)$, respectively.
\end{definition}

\noindent\textbf{Problem Statement:}
Given a time-varying optimization problem in \eqref{eq:log_barrier},
the objective is to design the rate of the optimization variables, $\dot{v}(t)$,
that guarantees a performance bound with pre-computable tube width $\rho > 0$,
such that $v(t) \in \mathcal{O}_{\nabla}(t, \rho) \subset \mathcal{T}_{\nabla}(\rho), \forall t \geq 0$  in the presence of the prediction error.

In the following,
we introduce two methods related to $\mathcal{L}_{1}$-AO:
i) continuous-time prediction-correction interior-point method
and ii) $\mathcal{L}_{1}$ adaptive control.

\subsection{Prediction-correction Interior-point method}
In~\cite{fazlyabPredictionCorrectionInteriorPointMethod2018},
a continuous-time prediction-correction interior-point (PCIP) method was proposed as follows:
\begin{equation}
\label{eq:ip_tvopt}
    \dot{v}(t)
    = -\nabla_{vv}\Phi(t)^{-1} \left(
        \nabla_{vt} \Phi(t)
        + P \nabla_{v} \Phi(t)
    \right),
\end{equation}
with $v(0) = v_{0}$,
where $P \succeq \beta I_{n_{v}}$ is the optimizer gain with $\beta > 0$, which determines the convergence rate.
The terms $\nabla_{vt}\Phi$ and $\nabla_{v} \Phi$ correspond to the \textit{prediction} and \textit{correction} terms,
which can be interpreted as \textit{feedforward} and \textit{feedback} terms in control~\cite{fazlyabPredictionCorrectionInteriorPointMethod2018}.
The prediction term tracks the changes in cost function over time,
and the correction term pushes the optimization variables to the optimal solution.
The update law forms a closed-loop dynamics
$\dot{\nabla}_{v} \Phi = - P \nabla_{v} \Phi$.
Considering the Lyapunov function $V(\nabla_{v}\Phi) := \frac{1}{2} \nabla_{v}\Phi^{\intercal} \nabla_{v} \Phi$,
it can be shown that $\dot{V} \leq - 2\beta V$ under the PCIP update law,
which implies the exponential convergence of the gradient to the origin~\cite{fazlyabPredictionCorrectionInteriorPointMethod2018}.

For online implementation,
a modified PCIP was proposed in~\cite{fazlyabPredictionCorrectionInteriorPointMethod2018}: if a prediction model with globally uniformly bounded prediction error is available,
i.e., $\lVert \tilde{\nabla}_{vt} \Phi (t, v) \rVert \leq \eta, \forall (t, v)$ for some $\eta > 0$,
the modified PCIP update law is given by
\begin{equation}
\label{eq:ip_tvopt_modified}
    \dot{v}(t)
    =
    -\nabla_{vv}^{-1}\Phi(t) \left(
        \hat{\nabla}_{vt} \Phi(t)
        + \frac{P\nabla_{v} \Phi(t)}{\max(\lVert \nabla_{v} \Phi(t) \rVert, \epsilon)}
    \right),
\end{equation}
with $v(0) = v_{0}$. The inequalities
$\beta > \eta $ and $\epsilon > 0$ imply
$\lim_{t \to \infty} \lVert \nabla_{v} \Phi (t, v) \rVert \leq \eta \epsilon / \beta$~\cite{fazlyabPredictionCorrectionInteriorPointMethod2018}.
The modified PCIP may suffer from several potential issues:
i) high gain is required to reduce the asymptotic bound,
which is undesirable for numerical stability,
and ii) the assumption that the prediction error is globally uniformly bounded may be unrealistic.

\subsection{\texorpdfstring{$\mathcal{L}_{1}$}{L1} Adaptive Control}
Consider the following system dynamics:
\begin{equation} \label{eq:main}
    \dot{x}(t) = f(x(t))
    + B(x(t)) ( u(t) + \sigma(t, x(t))),
\end{equation}
with $x(0)=x_{0}$,
where $x: \mathbb{R}_{\geq 0} \to X \subset \mathbb{R}^{n_{x}}$
and $u: \mathbb{R}_{\geq 0} \to U \subset \mathbb{R}^{n_{u}}$ denote the state and control input, respectively,
$f: X \to \mathbb{R}^{n_{x}}$
and $B: X \to \mathbb{R}^{n_{x} \times n_{u}}$ are known, 
and $\sigma: \mathbb{R}_{\geq 0} \times X \to \mathbb{R}^{n_{u}}$ is the matched uncertainty.
We assume $B(x)$ has full row rank for all $x \in X$.
We further assume the existence of a \emph{baseline control input} $u_{b}(t)$ such that $0 \in \mathbb{R}^{n_x}$ is a globally exponentially stable equilibrium of~\eqref{eq:main} with $u(t)=u_b(t)$ and $\sigma \equiv 0$.
In the presence of the uncertainty $\sigma \neq 0$, the input $u(t) = u_b(t)$ may not be able to stabilize~\eqref{eq:main}, leading to an unpredictable and unquantifiable divergence of $x(t)$.  
Hence, we design the input $u(t)$ as $u(t) = u_b(t) + u_a(t)$, where $u_a(t)$ denotes the \ellone adaptive control input.
The \ellone adaptive control (\ellone-AC) is a robust adaptive control methodology that decouples the adaptation loop from the control loop, which enables arbitrarily fast adaptation with guaranteed robustness margins~\cite{hovakimyan2010}.
The \ellone-AC has been successfully implemented
on NASA AirSTAR sub-scale aircraft~\cite{gregory_l1_2009}
and Calspan's variable-stability Learjet~\cite{ackerman_evaluation_2017}.
The \ellone-AC input $u_a(t)$ is defined through the following three components:
i) state predictor,
ii) uncertainty estimator (adaptation law),
and iii) low-pass filter.
Fig. \ref{fig:l1ac} illustrates the architecture of the \ellone-AC.
The state predictor is governed by the following equation:
\begin{equation}
\label{eq:l1ac_state_predictor}
    \dot{\hat{x}}(t)
    = A_{s} \tilde{x}(t)
    + f(x(t))
    + B(x(t)) (
    u(t) + \hat{\sigma}(t)
    ),
\end{equation}
with $\hat{x}(0) = x_{0}$,
where $A_{s} \in \mathbb{R}^{n_{x} \times n_{x}}$
is a
Hurwitz matrix,
$\hat{x}: \mathbb{R}_{\geq 0} \to \mathbb{R}^{n_{x}}$ is the predictor state,
$\hat{\sigma}: \mathbb{R}_{\geq 0} \to \mathbb{R}^{n_{u}}$ is the estimated uncertainty,
and $\tilde{x} := \hat{x} - x$.
The estimated uncertainty $\hat{\sigma}(t)$ can be updated using adaptation laws such as projection-based adaptation or piecewise-constant adaptation~\cite{hovakimyan2010,wuMathcal_1Quad2023a}.
The adaptive control input is then given by $u_{\text{a}}(s) = C(s) (-\hat{\sigma}(s))$,
where $C(s)$ is a low-pass filter with \emph{bandwidth} $\omega$ and satisfies $C(0)=1$.
The \ellone-AC input ensures  the existence of an \emph{a priori} $\rho > 0$ such that the state $x(t)$ of the uncertain system~\eqref{eq:main} with input $u(t) = u_b(t) + u_a(t)$ satisfies the \emph{uniform} bound $\lVert x \rVert_{\mathcal{L}_\infty} \leq \rho$.
The value of the bound $\rho > 0$ is determined by the choice of the filter bandwidth and the sampling period, $\omega$ and $T_s$, respectively. 
Further details on the design of \ellone-AC can be found in~\cite{hovakimyan2010,wuMathcal_1Quad2023a,wang_l1_2012,lakshmananSafeFeedbackMotion2020a}.

\begin{figure}
    \centering
    \subfigure[$\mathcal{L}_{1}$-AC]{\includegraphics[width=0.34\textwidth]{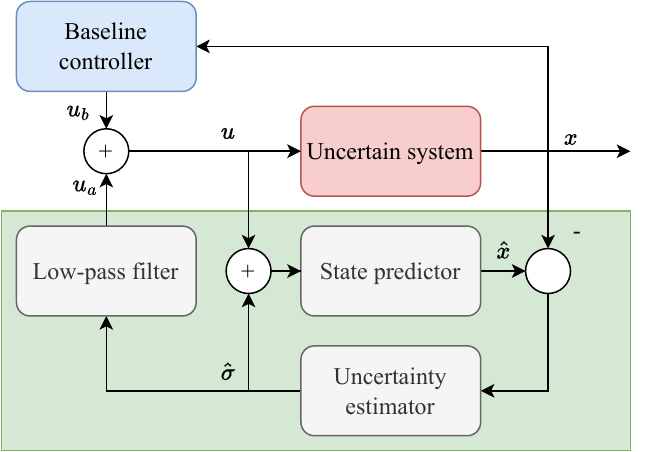}\label{fig:l1ac}}
    \subfigure[$\mathcal{L}_{1}$-AO (proposed)]{\includegraphics[width=0.37\textwidth]{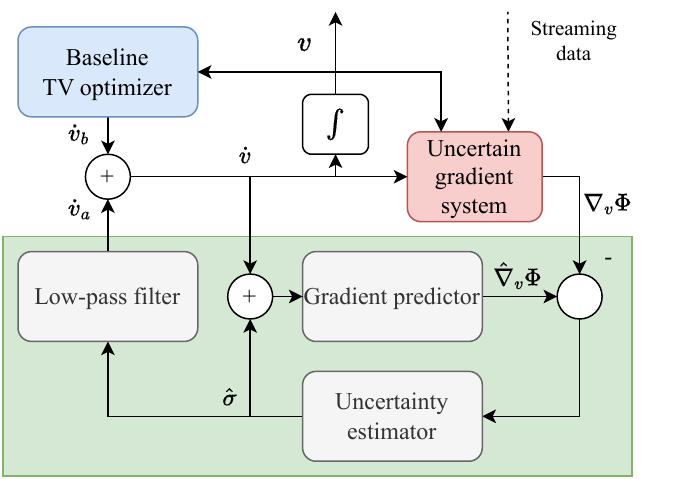}\label{fig:l1ao}}
    \caption{
    Architecture comparison between $\mathcal{L}_{1}$ adaptive control ($\mathcal{L}_{1}$-AC)
    and the proposed $\mathcal{L}_{1}$ adaptive optimization ($\mathcal{L}_{1}$-AO).
    Both architectures are designed to compensate for the uncertainty to recover the performance of the baseline controller or optimizer.
    (a) In $\mathcal{L}_{1}$-AC,
    the state predictor predicts the state of the uncertain (physical) system.
    (b) In $\mathcal{L}_{1}$-AO, the \textit{gradient system} (a virtual system) is considered.
    The rate of optimization variables, $\dot{v}$, is regarded as the control input,
    and the gradient predictor is used to predict the gradient of the cost function.
    }
    \label{fig:architectures}
\end{figure}
\section{\texorpdfstring{$\mathcal{L}_{1}$}{L1} Adaptive Optimization}
\label{sec:l1ao}
The main observation from \autoref{sec:preliminaries} is that a TV optimization problem with an inaccurate prediction model can be interpreted as an output regulation problem of an uncertain nonlinear system, where the output is the gradient.
This interpretation allows us to derive an adaptive TV optimization algorithm based on robust adaptive control.

\subsection{Algorithm and Architecture}
We propose an adaptive method for online TV optimization, termed {\it $\mathcal{L}_{1}$ adaptive optimization} ($\mathcal{L}_{1}$-AO).
We assume the existence of a \textit{baseline update law} $\dot{v}_{b}(t)$ such that the update law $\dot{v}(t) = \dot{v}_{b}(t)$ regulates the gradient to the origin in the absence of uncertainty (\eqref{eq:grad_dyn_est} with $\sigma \equiv 0$) as follows.
\begin{assumption}
\label{assumption:baseline_optimizer}
    Suppose there exists a Lyapunov function
    $V$ such that the following holds for the nominal gradient dynamics
    (\eqref{eq:grad_dyn_est} with $\sigma \equiv 0$)
    under the baseline update law $\dot{v}(t) = \dot{v}_{b}(t)$:
    \begin{equation}
    \label{eq:baseline_condition}
        \underline{\alpha} \lVert \nabla_{v}\Phi \rVert^{2}
        \leq V(\nabla_{v} \Phi)
        \leq \overline{\alpha} \lVert \nabla_{v}\Phi \rVert^{2},
        \quad
        \dot{V}
        \leq - 2\beta V,
    \end{equation}
    for some positive numbers $\underline{\alpha}$, $\overline{\alpha}$, and $\beta$.
\end{assumption}
The $\mathcal{L}_{1}$-AO consists of the following three components:
i) gradient predictor,
ii) uncertainty estimator (adaptation law),
and iii) low-pass filter.
Adopting the piece-wise constant adaptation law~\cite{wuMathcal_1Quad2023a},
each component can be described as
\begin{align*}
    \dot{\hat{\nabla}}_{v} \Phi(t)
    &= A_{s} \tilde{\nabla}_{v} \Phi(t)
    + \hat{\nabla}_{vt} \Phi(t)
    + \nabla_{vv} \Phi(t) \dot{v}(t) + h(t),
    \\
    h(i T_{s})
    &= \mu (A_{s}, T_{s}) \tilde{\nabla}_{v} \Phi(i T_{s}), \quad \forall i \in \mathbb{N}
    \\
    \hat{\sigma}(i T_{s})
    &= \nabla_{vv}\Phi^{-1}(i T_{s})
    h(i T_{s}), \quad \forall i \in \mathbb{N}
    \stepcounter{equation}\tag{\theequation}\label{eq:l1ao_alg}
    \\
    \dot{v}_{\text{a}}(s)
    &= C(s) (-\hat{\sigma}(s)),
\end{align*}
with $\hat{\nabla}_{v} \Phi(0) = \nabla_{v} \Phi(0, v_{0})$,
$h(0) = \hat{\sigma}(0) = 0$,
and $\dot{v} = \dot{v}_{b} + \dot{v}_{a}$,
where
\begin{equation}
    \mu(A_{s}, T_{s})
    := (A_{s}^{-1} (I - e^{A_{s} T_{s}}))^{-1} e^{A_{s} T_{s}},
\end{equation}
and $h(t) := h(i T_{s})$ and $\hat{\sigma}(t) := \hat{\sigma}(i T_{s}), \forall t \in [i T_{s}, (i+1) T_{s})$ for any $i \in \mathbb{N}_{0}$, represent the estimation of $\nabla_{vv} \Phi(t) \sigma(t)$ and $\sigma(t)$, respectively.
$A_{s} \in \mathbb{R}^{n_{v} \times n_{v}}$ is a diagonal Hurwitz matrix,
$T_{s} > 0$ is the sampling period of the adaptation,
and $\tilde{\nabla}_{v} \Phi := \hat{\nabla}_{v} \Phi - \nabla_{v} \Phi$.
Consider a low-pass filter $C(s) = \frac{\omega}{s+\omega}$ with a filter parameter $\omega > 0$.
From \eqref{eq:grad_dyn_est} and \eqref{eq:l1ao_alg},
the prediction error dynamics are:
\begin{equation}
\label{eq:prediction_error_dynamics}
    \dot{\tilde{\nabla}}_{v} \Phi(t)
    = A_{s} \tilde{\nabla}_{v} \Phi(t)
    + h(t) - \nabla_{vv} \Phi(t) \sigma(t).
\end{equation}
Figure \ref{fig:l1ao} illustrates the architecture of the $\mathcal{L}_{1}$-AO.

\subsection{Performance and Optimality Analysis}
In this section,
the performance and optimality of the proposed $\mathcal{L}_{1}$-AO algorithm in \eqref{eq:l1ao_alg} are analyzed.

For arbitrarily chosen $\epsilon > 0$, let
    $\rho := \sqrt{\overline{\alpha}/\underline{\alpha}} \Vert \nabla_{v} \Phi(0) \Vert + \epsilon$.
Next, we assume the following regularity condition.
\begin{assumption}
\label{assumption:tube}
    Suppose that there exist positive numbers $\Delta_{(\cdot)}$'s
    such that the following inequalities hold:
    \begin{align*}
        \stepcounter{equation}\tag{\theequation}\label{eq:v_dot_b_condition}
        \left\lVert \dot{v}_{b} (t, v) \right\rVert
        &\leq \Delta_{\dot{v}_{b}},              
        \left\lVert \frac{\partial V(\nabla_{v}\Phi(t, v))}{\partial \nabla_{v}\Phi} \right\rVert
        \leq \Delta_{\frac{\partial V}{\partial \nabla_{v}\Phi}},    
        \\
        \left\lVert \hat{\nabla}_{vt} \Phi(t, v) \right\rVert
        &\leq \Delta_{\hat{\nabla}_{vt}\Phi},
        \quad
        \left\lVert \tilde{\nabla}_{vt} \Phi(t, v) \right\rVert
        \leq \Delta_{\tilde{\nabla}_{vt}\Phi},
        \\
        \left\lVert \tilde{\nabla}_{vtt} \Phi(t, v) \right\rVert
        &\leq \Delta_{\tilde{\nabla}_{vtt}\Phi},
        \quad
        \left\lVert \tilde{\nabla}_{vvt} \Phi(t, v) \right\rVert
        \leq \Delta_{\tilde{\nabla}_{vvt}\Phi},
        \\       
        \left\lVert \nabla_{vv} \Phi (t, v) \right\rVert
        &\leq \Delta_{\nabla_{vv}\Phi},
        \quad
        \left\lVert \nabla_{vvt} \Phi(t, v) \right\rVert
        \leq \Delta_{\nabla_{vvt} \Phi},       
        \\
        \left\lVert \nabla_{vvv} \Phi(t, v) \right\rVert
        &\leq \Delta_{\nabla_{vvv} \Phi},       
    \end{align*}
    for all $t \geq 0$ and $ v \in \mathcal{T}_{\nabla}(\rho)$.
\end{assumption}
Based on the above assumption, let
\begin{align}
    \Delta_{\sigma}
    &:= m_{f}^{-1} \Delta_{\tilde{\nabla}_{vt} \Phi},
    \\
    \Delta_{\hat{\sigma}}
    &:=\sqrt{n_{v}} m_{f}^{-1} \Delta_{\nabla_{vv}\Phi} \Delta_{\sigma},
    \\
    \Delta_{\dot{v}}
    &:= \Delta_{\dot{v}_{b}} + \Delta_{\hat{\sigma}},
    \\
    \Delta_{\dot{\nabla}_{vv}\Phi}
    &:= \Delta_{\nabla_{vvt} \Phi} + \Delta_{\nabla_{vvv} \Phi} \Delta_{\dot{v}},
    \\
    \Delta_{\dot{\tilde{\nabla}}_{vt}\Phi}
    &:= \Delta_{\tilde{\nabla}_{vtt} \Phi} + \Delta_{\tilde{\nabla}_{vvt} \Phi} \Delta_{\dot{v}},
    \\
    \Delta_{\dot{\sigma}}
    &:= m_{f}^{-1} \left( \Delta_{\dot{\nabla}_{vv}\Phi}
    m_{f}^{-1} \Delta_{\tilde{\nabla}_{vt} \Phi}
    + \Delta_{\dot{\tilde{\nabla}}_{vv} \dot{\Phi}}
    \right),
    \\
    \zeta_{1}(\omega)
    &:= \Delta_{\frac{\partial V}{\partial \nabla_{v} \Phi}}
    \Delta_{\nabla_{vv} \Phi}
    \left(
        \frac{\Delta_{\sigma}}{\lvert 2\beta - \omega \rvert}
        + \frac{\Delta_{\dot{\sigma}}}{2\beta \omega}
    \right),
\end{align}
\begin{equation}
\begin{split}
     \zeta_{2}(A_{s})
    := 
 m_{f}^{-1} \sqrt{n_{v}}
 \Big(&
    (2\Delta_{\dot{\sigma}} + \underline{\lambda}(-A_{s}) \Delta_{\sigma} )
    \Delta_{\nabla_{vv} \Phi}
    \\
    &+
    \Delta_{\sigma}
    \Delta_{\dot{\nabla}_{vv} \Phi}
\Big),
\end{split}
\end{equation}
\begin{align}
    \zeta_{3}(\omega)
    &:= \Delta_{\sigma} \omega,
    \\
    \zeta_{4}(A_{s}, \omega)
    &:= \Delta_{\frac{\partial V}{\partial \nabla_{v} \Phi}}
    \Delta_{\nabla_{vv} \Phi}
    (\zeta_{2}(A_{s}) + \zeta_{3}(\omega))/(2\beta).
\end{align}

Choose the sampling time $T_{s} > 0$ and the filter bandwidth $\omega > 0$ such that the following inequalities hold:
\begin{equation}
\label{eq:conditions_sampling_time_and_filter_bandwidth}
\begin{split}
    \underline{\alpha} \rho^{2}
    &> \zeta_{1}(\omega) + V_{0},
    \\
    T_{s}
    &\leq (\underline{\alpha}\rho^{2} - \zeta_{1}(\omega) - V_{0}) / \zeta_{4},
\end{split}
\end{equation}
where $V_{0} := V(\nabla_{v}\Phi(0, v(0)))$.
Note that $\zeta_{1}(\omega) \to +0$ as $\omega \to +\infty$.
Therefore, from the definition of $\rho$ and \autoref{assumption:baseline_optimizer},
the conditions in \eqref{eq:conditions_sampling_time_and_filter_bandwidth} can be satisfied by choosing sufficiently large $\omega$ and small $T_{s}$.

We now present the main result.
\begin{theorem}
    \label{thm:l1ao}
    Suppose \autoref{assumption:baseline_optimizer}, \autoref{assumption:tube}, and \eqref{eq:conditions_sampling_time_and_filter_bandwidth} hold.
    Then, under the $\mathcal{L}_{1}$ adaptive optimization algorithm in \eqref{eq:l1ao_alg},
   we have
    \begin{equation}
    v(t)
    \in \mathcal{O}_{\nabla}(t, \rho)
    \subset \mathcal{O}(t, \rho / m_{f}), \forall t \geq 0.
    \end{equation}
    Furthermore, for any $t_{1} > 0$, we have
    \begin{equation}
    \label{eq:UUB}
    \begin{aligned}
    v(t)
    &\in \mathcal{O}_{\nabla} \left(
    t, \sqrt{
    (e^{-2\beta t_{1}} V_{0} + \zeta_{1} + \zeta_{4} T_{s})
    / \underline{\alpha}}
    \right)
    \\
    &\subset
    \mathcal{O} \left(
    t, m_{f}^{-1} \sqrt{
    (e^{-2\beta t_{1}} V_{0} + \zeta_{1} + \zeta_{4} T_{s})
    / \underline{\alpha}}
    \right),
    \end{aligned}
    \end{equation}
    for all $t \geq t_{1}$.
\end{theorem}
\ifisnotpreprint
\else
\begin{proof}
    See Appendix \ref{sec:proof_of_tube}.
\end{proof}
\fi
The main theorem implies that the gradient and the error in optimization variables are uniformly bounded for all $t \geq 0$ and uniformly ultimately bounded for all $t \geq t_{1}$ for any $t_{1} > 0$.
The following corollary shows the boundedness of the optimality gap.
\begin{corollary}
\label{corollary:cost_bound}
    Under the assumptions of \autoref{thm:l1ao},  for all $t \geq 0$
   we have:
    \begin{equation}
        0 \leq \Phi(t, v(t)) - \Phi(t, v^{\star}(t)) \leq \rho^{2} / m_{f}.
    \end{equation}
    Also, for any $t_{1} > 0$, we have for all $t \geq t_{1}$:
    \begin{equation}
        0 \leq \Phi(t, v(t)) - \Phi(t, v^{\star}(t)) \leq \frac{e^{-2\beta t_{1}} V_{0} + \zeta_{1} + \zeta_{4} T_{s}}{m_{f} \underline{\alpha}}.
    \end{equation}
\end{corollary}
\ifisnotpreprint
\else
\begin{proof}
    See Appendix \ref{sec:proof_of_cost_tube}.
\end{proof}
\fi

\begin{remark}
    \autoref{assumption:baseline_optimizer},
    and \eqref{eq:v_dot_b_condition}
    in \autoref{assumption:tube}
    hold for the PCIP method in \eqref{eq:ip_tvopt}
    with
    $\underline{\alpha} = \overline{\alpha} = 1/2$,
    $\beta = \underline{\lambda}(P)$,
    $\Delta_{\dot{v}_{b}} = m_{f}^{-1} (\Delta_{\hat{\nabla}_{vt} \Phi} + \lVert P \rVert \rho )$,
    and $\Delta_{\frac{\partial V}{\partial \nabla_{v} \Phi}} = \rho$.
    The \autoref{assumption:tube} is for $v \in \mathcal{T}_{\nabla}(\rho)$, \textup{not} assuming $v(t) \in \mathcal{T}_{\nabla}(\rho),
    \forall t \geq 0$.
\end{remark}

\section{NUMERICAL SIMULATIONS}
\label{sec:numerical_simulation}
In this section,
numerical simulations are performed for time-varying optimization problems to demonstrate the effectiveness of the proposed method.
\ifisnotpreprint
\else
Simulation settings can be found in Appendix \ref{sec:simulation_settings}.
\fi

\subsection{Example 1: Constrained TV optimization}
\label{sec:example1}
Consider the following TV optimization problem:
\begin{align}
    \min_{v} f_{0}(t, v)
    &= \frac{1}{2} v^{2}
    \\
    \label{eq:sim_constraint}
    \text{subject to} \quad f_{1}(t, v)
    &= v + d(t) \leq 0,
\end{align}
with external streaming data $d(t) = 3 \sin(3t)$.
The problem is approximated by a log-barrier function as follows:
\begin{equation}
    \min_{v} \Phi(t, v)
    = \frac{1}{2} v^{2}
    - \log(-(v + d(t))),
\end{equation}
which implies $\nabla_{vt} \Phi(t, v) = \dot{d}(t)/(v + d(t))^{2}$.
The nominal prediction model assumes that $d(t)$ slowly varies over time, i.e., $\dot{d}(t) \approx 0$, which implies $\hat{\nabla}_{vt} \Phi(t, v) = 0$.

\noindent \textbf{Results}:
The simulation results are illustrated in \autoref{fig:sim_result_ex2}.
As shown in Fig. \ref{fig:sim_pcip},
the modified PCIP method with gain $P=10$ could not compensate for the prediction inaccuracies effectively and violated the inequality constraint, causing early termination of the simulation (denoted by a magenta star).
The same simulation with gain $P=10^{3}$ did not violate the constraint and tracked the optimal solution.
However, this induced an extremely large update rate at the beginning phase,
undesirable for numerical stability.
Note that the modified PCIP method with gains larger than $P = 10^{3}$ caused failures of the numerical integration solvers (immediate termination after one time step).
On the other hand,
as shown in Fig. \ref{fig:sim_l1ao},
the proposed $\mathcal{L}_{1}$-AO successfully tracks the optimal solution without any excessive update rate.
The optimization variable is confined within a green tube $\mathcal{T}(\rho)$ with $\rho = 0.28$,
satisfying the constraint in the presence of prediction inaccuracies.

\autoref{tab:elapsed_time} shows each method's mean elapsed times (total elapsed time divided by the number of time steps).
ECOS~\cite{domahidiECOSSOCPSolver2013}, a time-invariant solver, results in the mean elapsed time larger than the time step $\Delta t=10^{-3}$, which is not applicable to real-time scenarios.
On the other hand, modified PCIP and $\mathcal{L}_{1}$-AO, which are time-varying methods, show significantly small mean elapsed times.
The mean elapsed time of $\mathcal{L}_{1}$-AO is larger than that of modified PCIP due to extra computation for adaptation.
The mean elapsed time of $\mathcal{L}_{1}$-AO is approximately $1/120$ of that of ECOS and $1/50$ of that of time step $\Delta t$, making it
desirable for real-time scenarios.

\begin{figure}[!t]
    \centering
    \subfigure[Modified PCIP]{\includegraphics[width=0.42\textwidth]{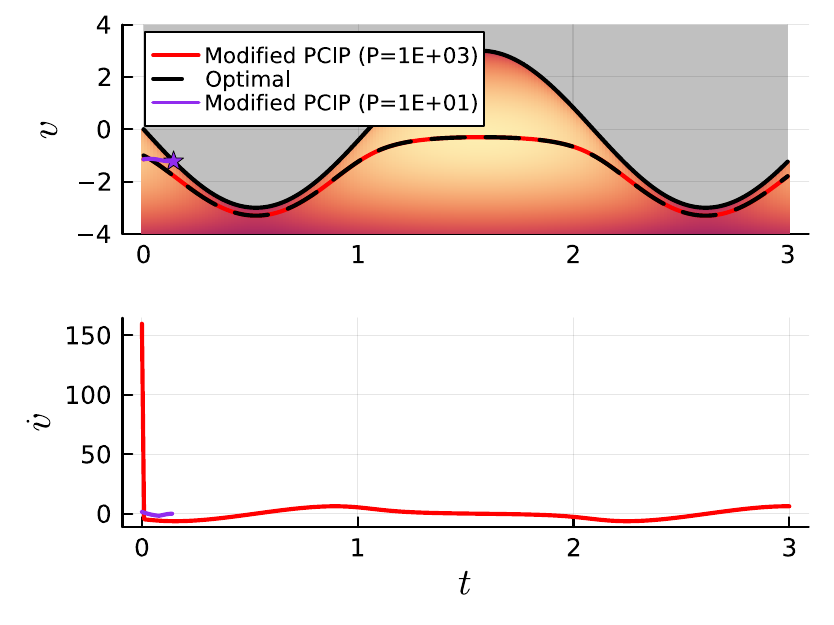}\label{fig:sim_pcip}}
    \subfigure[$\mathcal{L}_{1}$-AO (baseline: PCIP)]{\includegraphics[width=0.42\textwidth]{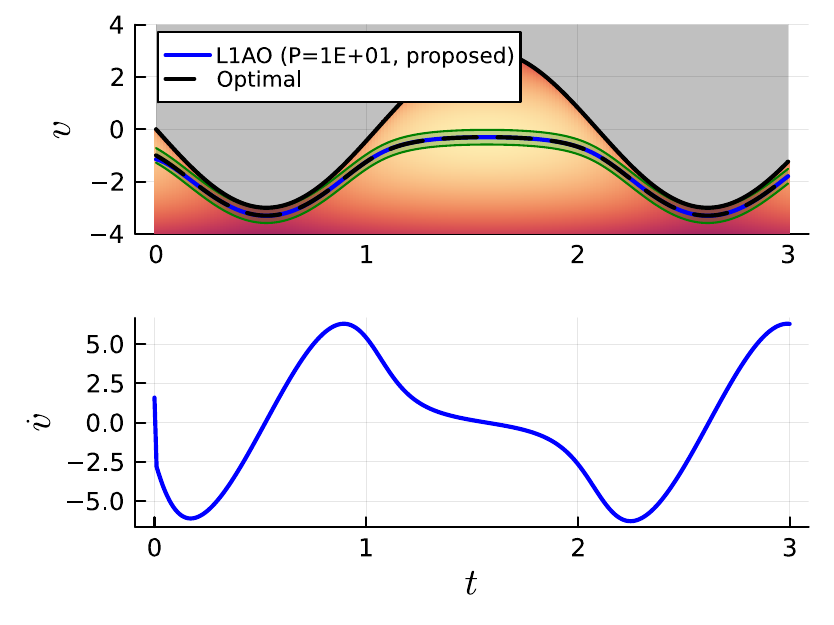}\label{fig:sim_l1ao}}
    \caption{
    Example 1: Simulation result.
    The gray area denotes the TV constraint.
    The heatmap indicates the cost function.
    (a) The modified PCIP with $P=10$ violated the TV constraint (magenta star).
    The modified PCIP with a high gain $P=10^{3}$ induced a huge update rate at the initial phase,
    undesirable for numerical stability.
    (b) The $\mathcal{L}_{1}$-AO updates the optimization variable without an abrupt update.
    The optimization variable is confined within a green tube,
    preventing constraint violation.
    }
    \label{fig:sim_result_ex2}
\end{figure}
\begin{table}[h!]
    \caption{Example 1: Mean elapsed time for a time step}
    \centering
    \begin{tabular}{c|c|c}
    \hline
    ECOS & Modified PCIP$^{*}$ ($P=10^{3}$) & $\mathcal{L}_{1}$-AO (baseline: PCIP) \\
    \hline
    $2.2837$ ms & $0.0063$ ms & $0.0187$ ms
    \\
    \hline
    \multicolumn{3}{l}{%
    \begin{minipage}{8cm}%
        $^{*}$: Modified PCIP with $P=10$ violated constraint
    \end{minipage}%
        }\\ 
    \end{tabular}
    \label{tab:elapsed_time}
\end{table}

\subsection{Example 2: Robot navigation}
\label{sec:example2}
We consider a robot navigation problem borrowed from~\cite[Section IV]{fazlyabPredictionCorrectionInteriorPointMethod2018}.
In the robot navigation task,
a projected goal is used for collision avoidance, which minimizes the following cost function:
\begin{align*}
    \Phi(t, x)
    = \frac{1}{2} \Vert x - x_{d} \Vert^{2} - \frac{1}{c(t)} \sum_{i=1}^{m} \log(b_{i}(x_{c}) - a_{i}(x_{c})^{\intercal} x),
\end{align*}
where $x \in \mathbb{R}^{2}$ is the optimization variable.
The positions of robot and target at time $t$, $x_{c}(t) \in \mathbb{R}^{2}$ and $x_{d}(t) \in \mathbb{R}^{2}$, respectively, can be regarded as streaming data.
Other details can be found in~\cite{fazlyabPredictionCorrectionInteriorPointMethod2018,arslanExactRobotNavigation2016}.
The nominal prediction model assumes that the target is static, i.e., $\dot{x}_{d}(t) \approx 0$.
Optimization algorithms produce the estimate $\hat{x}(t)$ of the projected goal, corresponding to the optimization variable $v(t)$ in this paper.

\noindent \textbf{Results}:
The simulation results are illustrated in~\autoref{fig:sim_2}.
As shown in Fig.~\ref{fig:sim_2}, PCIP exhibited a large tracking error due to the prediction inaccuracies.
Modified PCIP showed relatively low tracking error due to the high gain.
For both PCIP and modified PCIP,
the augmentation of $\mathcal{L}_{1}$-AO significantly improved the tracking performance.
The simulation results demonstrate that the $\mathcal{L}_{1}$-AO can effectively compensate for the prediction inaccuracies in online TV optimization.

\begin{figure}
    \centering
    \subfigure[$t=\frac{1}{4}t_{f}$]{\includegraphics[width=0.44\linewidth]{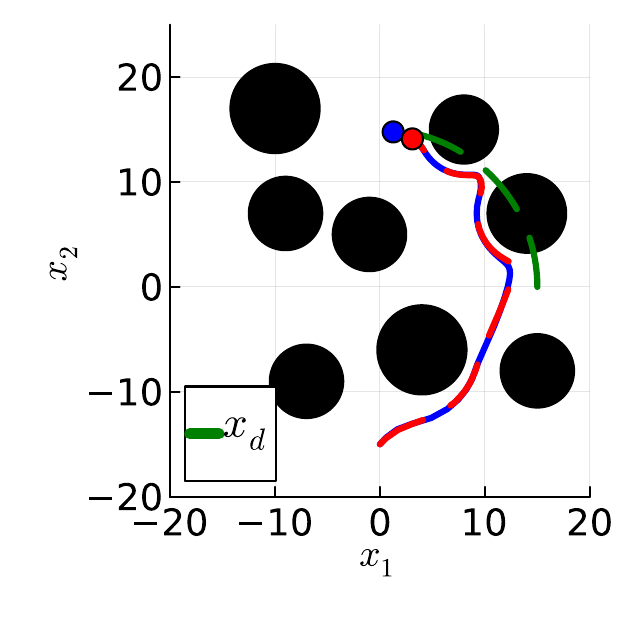}}\label{fig:sim_2_2d_1}
    \subfigure[$t=\frac{2}{4}t_{f}$]{\includegraphics[width=0.44\linewidth]{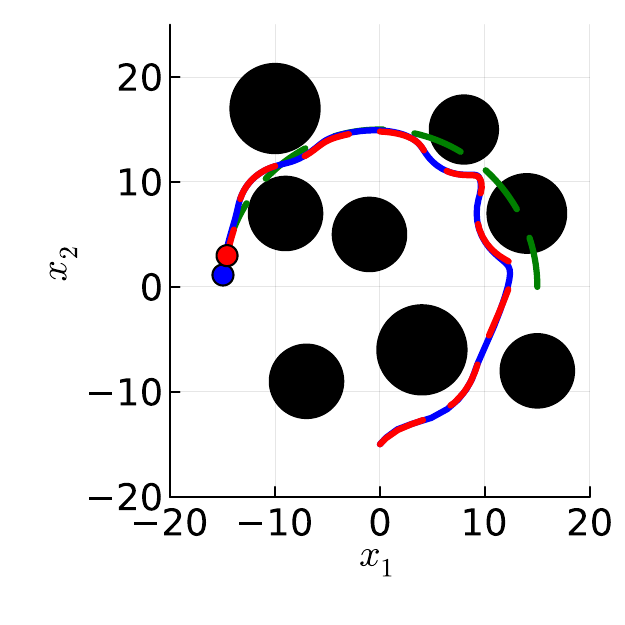}}\label{fig:sim_2_2d_2}
    \subfigure[$t=\frac{3}{4}t_{f}$]{\includegraphics[width=0.44\linewidth]{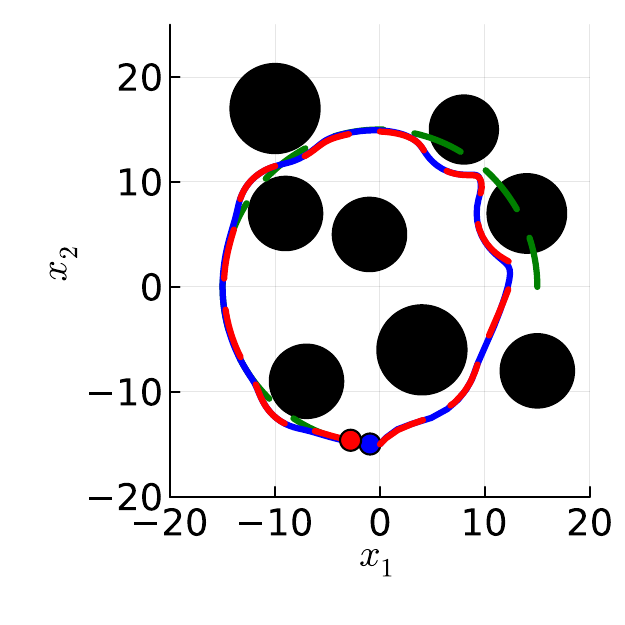}}\label{fig:sim_2_2d_3}
    \subfigure[$t=t_{f}$]{\includegraphics[width=0.44\linewidth]{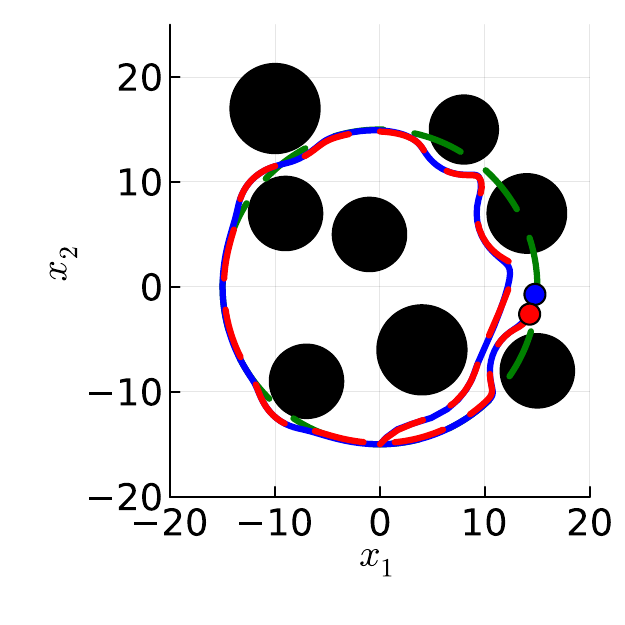}}\label{fig:sim_2_2d_4}
    \subfigure[]{\includegraphics[width=0.44\linewidth]{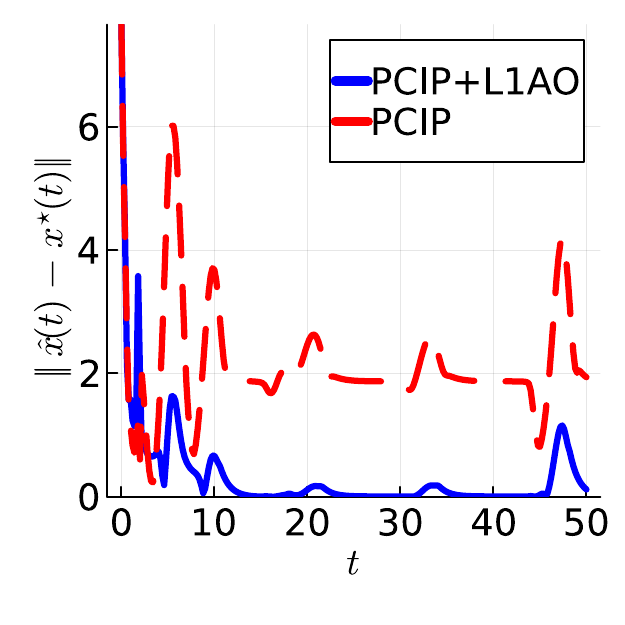}}\label{fig:sim_2_dist_1}
    \subfigure[]{\includegraphics[width=0.44\linewidth]{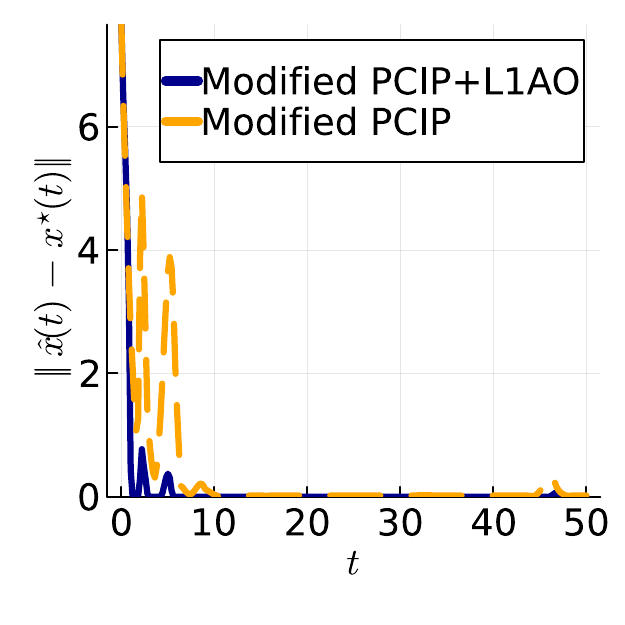}}\label{fig:sim_2_dist_2}
    \caption{
    Example 2: 
    (a)-(d) 2D trajectories of a robot. For visualization, only the results of PCIP and $\mathcal{L}_{1}$-AO with the baseline of PCIP are depicted.
    (e)-(f) Error in optimization variable: (e) PCIP vs $\mathcal{L}_{1}$-AO (baseline: PCIP), (f) Modified PCIP vs $\mathcal{L}_{1}$-AO (baseline: Modified PCIP).
    }
    \label{fig:sim_2}
\end{figure}

\section{CONCLUSIONS}
\label{sec:conclusions}
In this study,
we proposed $\mathcal{L}_{1}$ adaptive optimization ($\mathcal{L}_{1}$-AO),
an adaptive method for online time-varying (TV) convex optimization.
The $\mathcal{L}_{1}$-AO has a baseline TV optimizer, which updates the optimization variables with a prediction model.
The prediction model leverages the underlying temporal structure of TV problems,
allowing efficient TV optimization.
However,
the prediction model can be inaccurate in online implementation,
degrading the performance of TV optimizers.
The $\mathcal{L}_{1}$-AO estimates and compensates for the uncertainty from the prediction inaccuracies,
providing robustness certification with a pre-computable tube width.
Numerical simulation results showed that the proposed $\mathcal{L}_{1}$-AO can solve the TV optimization efficiently and reliably in the presence of prediction inaccuracies compared to existing time-invariant and non-adaptive TV methods.

In TV optimization, the streaming data may be influenced by the decision (optimization variable) through a physical dynamical system,
which forms a closed-loop feedback with the gradient dynamics.
Future work includes
analysis with such feedback,
relaxing strong convexity for a wider class of problems,
and more realistic applications of real-time TV optimization including robot perception and motion planning.

\bibliography{ref}
\bibliographystyle{ieeetr}

\ifisnotpreprint
\else
\appendix

\subsection{Lemmas for \autoref{thm:l1ao}}
In this section, we present lemmas for \autoref{thm:l1ao}.

\begin{lemma}
\label{lemma:ball_tube}
    For all $t \geq 0$ and $\rho >0$,
    $\mathcal{O}_{\nabla}(t, \rho) \subset \mathcal{O}(t, \rho / m_{f})$
    and
    $\mathcal{T}_{\nabla}(\rho) \subset \mathcal{T}(\rho / m_{f})$.
\end{lemma}
\begin{proof}
    Using mean-value theorem with respect to the second argument of $\nabla_{v}\Phi(t, v)$ around $v^{\star}(t)$ implies
    \begin{equation*}
    \nabla_{v} \Phi (t, v)
    - \nabla_{v} \Phi (t, v^{\star}(t))
    = \nabla_{vv}\Phi(t, \eta(t)) (v - v^{\star}(t)),
    \end{equation*}
    where $\eta(t)$ is a convex combination of $v$ and $v^{\star}(t)$.
    Then,
    \begin{equation*}
    \begin{split}
        \lVert v - v^{\star}(t) \rVert
        &= \lVert \nabla_{vv}^{-1}\Phi(t, \eta(t)) \nabla_{v} \Phi (t, v) \rVert
        \\
        &\leq \lVert \nabla_{vv}^{-1}\Phi(t, \eta(t)) \rVert \lVert \nabla_{v} \Phi (t, v) \rVert
        \\
        &\leq m_{f}^{-1} \lVert \nabla_{v}\Phi(t, v) \rVert,
    \end{split}
    \end{equation*}
    since $\nabla_{v} \Phi (t, v^{\star}(t)) = 0$
    and $\lVert \nabla_{vv}^{-1}\Phi(t, v) \rVert \leq m_{f}^{-1}$
    from the uniform strong convexity of $\Phi(t, \cdot)$,
    which concludes the proof.   
\end{proof}
\begin{lemma}
\label{lemma:uncertainty_est_bound}
    If $v(t) \in \mathcal{T}_{\nabla}(\rho), \forall t \in [0, \tau]$ for some $\tau > 0$,
    then the following relationships hold:
    \begin{align}
        \label{eq:sigma_hat_bound}
        \lVert \hat{\sigma} \rVert_{\mathcal{L}_{\infty}}^{[0, \tau]}
        &\leq \Delta_{\hat{\sigma}},
        \\
        \label{eq:v_dot_a_bound}
        \lVert \dot{v}_{a} \rVert_{\mathcal{L}_{\infty}}^{[0, \tau]}
        &\leq \Delta_{\hat{\sigma}}.       
        \\
        \label{eq:sigma_dot_bound}
        \lVert \dot{\sigma} \rVert_{\mathcal{L}_{\infty}}^{[0, \tau]}
        &\leq \Delta_{\dot{\sigma}}.
    \end{align}
\end{lemma}
\begin{proof}
Proof of \eqref{eq:sigma_hat_bound}:
It is sufficient to show that $\lVert \hat{\sigma} (i T_{s}) \rVert \leq \Delta_{\hat{\sigma}}, \forall i \in \{0, \ldots, \lfloor \tau/ T_{s} \rfloor \}$.
First, $\hat{\sigma}(0) = 0$.
Integrating \eqref{eq:prediction_error_dynamics} from $i T_{s}$ to $(i+1) T_{s}$ implies
\begin{align*}
    \tilde{\nabla}_{v} \Phi((i+1) T_{s})
    &= e^{A_{s} T_{s}} \tilde{\nabla}_{v} \Phi(i T_{s})
    \stepcounter{equation}\tag{\theequation}\label{eq:pred_err_integral_original}
    \\
    + \int_{i T_{s}}^{(i+1) T_{s}} & e^{A_{s} ((i+1) T_{s} - t)} (h(i T_{s}) -  \nabla_{vv} \Phi(t) \sigma(t)) dt,
\end{align*}
for all $i \in \mathbb{N}_{0}$.
It can be deduced from \eqref{eq:l1ao_alg} and \eqref{eq:pred_err_integral_original} that the following holds:
\begin{align*}
    \tilde{\nabla}_{v} \Phi((i+1) T_{s})&
    \\
    = -\int_{i T_{s}}^{(i+1) T_{s}} &e^{A_{s} ((i+1) T_{s} -t)} \nabla_{vv} \Phi(t) \sigma(t) dt,
    \stepcounter{equation}\tag{\theequation}\label{eq:pred_err_integral}
\end{align*}
for all $i \in \mathbb{N}$.
Since $A_{s}$ is a diagonal matrix,
applying the mean-value theorem element-wise yields
\begin{equation}
\label{eq:tmp2}
    \tilde{\nabla}_{v} \Phi((i+1) T_{s})
    = A_{s}^{-1} (e^{A_{s} T_{s}} - I)
    (\nabla_{vv} \Phi \sigma)^{*},
\end{equation}
for all $i \in \mathbb{N}_{0}$,
where
$(\nabla_{vv}\Phi \sigma)_{j}^{*} := (\nabla_{vv}\Phi(\tau_{j}) \sigma(\tau_{j}))_{j}$
with time indices $\tau_{j} \in (iT_{s}, (i+1)T_{s})$ induced by the mean-value theorem for $j$-th component.
Also, it is straightforward from the definition of $\sigma$ that
\begin{equation}
\label{eq:uncertainty_bound}
    \lVert \sigma(t, v(t)) \rVert
    \leq m_{f}^{-1} \Delta_{\tilde{\nabla}_{vt} \Phi} =: \Delta_{\sigma}, \forall t \in [0, \tau].
\end{equation}
For all $i \in \mathbb{N}_{0}$,
substituting \eqref{eq:l1ao_alg} into \eqref{eq:tmp2} yields
\begin{align*}
    &\lVert \hat{\sigma}((i+1)T_{s}) \rVert
    = \lVert
        \nabla_{vv} \Phi^{-1}((i+1)T_{s})
        e^{A_{s} T_{s}}
        (\nabla_{vv}\Phi \sigma)^{*}
    \rVert
    \\
    &\leq m_{f}^{-1}
    \lVert e^{A_{s} T_{s}} \rVert
    \lVert (\nabla_{vv}\Phi \sigma)^{*} \rVert
    \leq  m_{f}^{-1}
    \lVert (\nabla_{vv}\Phi \sigma)^{*} \rVert
    \\
    &\leq  m_{f}^{-1}
    \sqrt{n_{v}} \lVert (\nabla_{vv}\Phi \sigma)^{*} \rVert_{\infty}
    = m_{f}^{-1}
    \sqrt{n_{v}} \lvert (\nabla_{vv}\Phi \sigma)_{j^{*}}^{*} \rvert
    \\
    &\leq m_{f}^{-1}
    \sqrt{n_{v}} \lVert \nabla_{vv}\Phi(\tau_{j^{*}}) \sigma(\tau_{j^{*}}) \rVert
    \\
    &\leq m_{f}^{-1}
    \sqrt{n_{v}} \lVert \nabla_{vv}\Phi(\tau_{j^{*}}) \rVert \lVert \sigma(\tau_{j^{*}}) \rVert   
    \\
    \stepcounter{equation}\tag{\theequation}\label{eq:tmp1}
    &\leq m_{f}^{-1}
    \sqrt{n_{v}} \Delta_{\nabla_{vv} \Phi} \Delta_{\sigma}
    =: \Delta_{\hat{\sigma}},
\end{align*}
where $j^{*} := \argmax_{j} \lvert (\nabla_{vv} \Phi \sigma)_{j}^{*} \rvert$.
Therefore, $\lVert \hat{\sigma} (i T_{s}) \rVert \leq \Delta_{\hat{\sigma}},
\forall i \in \{0, \ldots, \lfloor \tau/ T_{s} \rfloor \}$.

Proof of \eqref{eq:v_dot_a_bound}:
By \cite[Lemma A.7.1]{hovakimyan2010}, we have
\begin{equation}
\lVert \dot{v}_{a} \rVert_{\mathcal{L}_{\infty}}^{[0, \tau]}
\leq \lVert C(s) \rVert_{\mathcal{L}_{1}}
    \lVert \hat{\sigma} \rVert_{\mathcal{L}_{\infty}}^{[0, \tau]}
= \lVert \hat{\sigma} \rVert_{\mathcal{L}_{\infty}}^{[0, \tau]}
\leq \Delta_{\hat{\sigma}},
\end{equation}
with the filter $C$ with unity DC gain.

Proof of \eqref{eq:sigma_dot_bound}:
First, $\lVert \dot{v}(t) \rVert \leq \lVert \dot{v}_{b}(t) \rVert + \lVert \dot{v}_{a}(t) \rVert \leq \Delta_{\dot{v}_{b}} + \Delta_{\hat{\sigma}} =: \Delta_{\dot{v}}, \forall t \in [0, \tau]$.
Then, by definition, we have
\begin{align*}
    \dot{\sigma}(t, v(t))
    \stepcounter{equation}\tag{\theequation}
    &= \frac{d}{dt} (\nabla_{vv} \Phi^{-1} \tilde{\nabla}_{vt} \Phi)
    \\
    &= -\Big(
        \frac{d \nabla_{vv}  \Phi^{-1}}{dt}  \tilde{\nabla}_{vt} \Phi
        + \nabla_{vv} \Phi^{-1} \frac{d \tilde{\nabla}_{vt}\Phi}{dt}
    \Big)
    \\
    &= \nabla_{vv} \Phi^{-1} \Big(
        \frac{d {\nabla}_{vv} \Phi}{d t}
        \nabla_{vv} \Phi^{-1} \tilde{\nabla}_{vt}\Phi
        - \frac{d \tilde{\nabla}_{vt} \Phi}{d t}
    \Big),
\end{align*}
and
\begin{align}
    &\left \Vert \frac{d \nabla_{vv} \Phi (t, v(t))}{d t} \right \Vert
    = \Vert \nabla_{vvt} \Phi(t) + \nabla_{vvv} \Phi(t) \dot{v}(t) \Vert
    \\
    &\leq \Delta_{\nabla_{vvt} \Phi} + \Delta_{\nabla_{vvv} \Phi} \Delta_{\dot{v}} =: \Delta_{\dot{\nabla}_{vv} \Phi},
    \\
    &\left \Vert \frac{d \tilde{\nabla}_{vt} \Phi (t, v(t))}{d t} \right \Vert
    = \Vert \tilde{\nabla}_{vtt} \Phi(t) + \tilde{\nabla}_{vvt}\Phi(t) \dot{v}(t) \Vert
    \\
    &\leq \Delta_{\tilde{\nabla}_{vtt}\Phi} + \Delta_{\tilde{\nabla}_{vvt}\Phi} 
 \Delta_{\dot{v}} =: \Delta_{\dot{\tilde{\nabla}}_{vt} \Phi}.
\end{align}
This implies
\begin{equation}
\begin{aligned}
    \lVert \dot{\sigma}(t, v(t)) \rVert
    &\leq
    m_{f}^{-1} (
    \Delta_{\dot{\nabla}_{vv} \Phi} m_{f}^{-1} \Delta_{\tilde{\nabla}_{vt} \Phi}
    + \Delta_{\dot{\tilde{\nabla}}_{vt} \Phi}
    )
    \\
    &=: \Delta_{\dot{\sigma}},
\end{aligned}
\end{equation}
which concludes the proof.
\end{proof}

\begin{lemma}
\label{lemma:bounded_estimation_error}
    If $v(t) \in \mathcal{O}_{\nabla}(t, \rho), \forall t \in [0, \tau]$ for some $\tau > 0$,
    then the error in uncertainty estimation $\tilde{\sigma}(t) := \hat{\sigma}(t) - \sigma(t, v(t))$ is bounded as
    \begin{equation}
        \lVert \tilde{\sigma}(t) \rVert
        \leq \begin{cases}
            \Delta_{\sigma}, & \forall t \in [0, T_{s}),
            \\
            \zeta_{2}(A_{s}) T_{s}, & \forall t \in [T_{s}, \tau].
        \end{cases}
    \end{equation}
\end{lemma}
\begin{proof}
First, $\hat{\sigma}(0) = 0$ implies $\lVert \tilde{\sigma}(t) \rVert = \lVert \sigma(t) \rVert \leq \Delta_{\sigma}, \forall t \in [0, T_{s})$.
For all $t \in [(i+1) T_{s}, (i+2) T_{s}), \forall i \in \mathbb{N}_{0}$,
\begin{equation}
\label{eq:tmp57}
\begin{split}
    &\Vert \tilde{\sigma} (t) \Vert
    = \Vert \hat{\sigma}(t) - \sigma(t) \Vert
    \\
    &= \Vert
        \nabla_{vv}\Phi^{-1} ((i+1)T_{s})
        e^{A_{s} T_{s}}(\nabla_{vv} \Phi \sigma)^{*} - \sigma(t)
    \Vert
    \\
    &\leq m_{f}^{-1}
    \Vert
        e^{A_{s} T_{s}}(\nabla_{vv} \Phi \sigma)^{*} - \nabla_{vv}\Phi ((i+1)T_{s})\sigma(t)
    \Vert.
\end{split}
\end{equation}
Similar to the proof of \autoref{lemma:uncertainty_est_bound}, using the mean-value theorem,
the following inequality can be derived from \eqref{eq:tmp57}:
\begin{align*}
    \Vert \tilde{\sigma} (t) \Vert
    \leq& m_{f}^{-1} \sqrt{n_{v}}
    \stepcounter{equation}\tag{\theequation}
    \\
    &\Vert
        e^{A_{s} T_{s}}\nabla_{vv} \Phi(\tau_{j^{\#}}) \sigma(\tau_{j^{\#}}) - \nabla_{vv}\Phi ((i+1)T_{s})\sigma(t)
    \Vert,
    \\
    \leq& m_{f}^{-1} \sqrt{n_{v}} \left(
        \chi_{1}(t) + \chi_{2}(t) + \chi_{3}(t)
    \right),
\end{align*}
where $j^{\#} := \argmax_{j} \lvert (
e^{A_{s} T_{s}}\nabla_{vv} \Phi(\tau_{j}) \sigma(\tau_{j}) - \nabla_{vv}\Phi ((i+1)T_{s})\sigma(t)
)_{j} \rvert $,
and
\begin{align*}
    \chi_{1}(t)
    := \Vert 
        &\nabla_{vv}\Phi ((i+1)T_{s}) (\sigma(t) - \sigma(\tau_{j^{\#}}))
    \Vert,
    \\
    \chi_{2}(t)
    := \Vert
        & \left(
            \nabla_{vv} \Phi((i+1) T_{s}) - \nabla_{vv} \Phi(\tau_{j^{\#}})
        \right) \sigma(\tau_{j^{\#}})   
    \Vert,
    \\
    \chi_{3}(t)
    := \Vert
        &(I_{n_{v}} - e^{A_{s} T_{s}}) \nabla_{vv} \Phi(\tau_{j^{\#}}) \sigma(\tau_{j^{\#}})
    \Vert.
    \stepcounter{equation}\tag{\theequation}
\end{align*}
From \autoref{lemma:uncertainty_est_bound},
one can deduce the following inequalities,
\begin{align*}
    \chi_{1}(t)
    &\leq \Delta_{\nabla_{vv} \Phi} \lVert \sigma(\tau_{j^{\#}}) - \sigma(t) \rVert
    \leq \Delta_{\nabla_{vv} \Phi} \int_{\tau_{j}^{\#}}^{t} \lVert \dot{\sigma}(\tau) \rVert d\tau
    \\
    &\leq \Delta_{\nabla_{vv} \Phi} \Delta_{\dot{\sigma}} (t - \tau_{j}^{\#})
    \leq 2 \Delta_{\nabla_{vv} \Phi} \Delta_{\dot{\sigma}} T_{s},
    \stepcounter{equation}\tag{\theequation}
    \\
    \chi_{2}(t)
    &\leq \Delta_{\sigma} \left\Vert
        \left(
            \nabla_{vv} \Phi((i+1) T_{s}) - \nabla_{vv} \Phi(\tau_{j^{\#}})
        \right)
    \right\Vert
    \\
    &\leq \Delta_{\sigma}
    \int_{\tau_{j^{\#}}}^{(i+1)T_{s}} \left\Vert \frac{d \nabla_{vv}\Phi }{d t} \right\Vert d t
    \\
    &\leq \Delta_{\sigma} \Delta_{\dot{\nabla}_{vv} \Phi} ((i+1) T_{s} - \tau_{j^{\#}})
    \leq \Delta_{\sigma} \Delta_{\dot{\nabla}_{vv} \Phi} T_{s},
    \\
    \chi_{3}(t)
    &\leq \Delta_{\sigma} \Vert
        (I_{n_{v}} - e^{A_{s} T_{s}} ) \nabla_{vv} \Phi(\tau_{j^{\#}})
    \Vert
    \\
    &\leq \Delta_{\sigma} \Delta_{\nabla_{vv} \Phi} \Vert
        I_{n_{v}} - e^{A_{s} T_{s}}
    \Vert \leq \Delta_{\sigma} \Delta_{\nabla_{vv} \Phi} \underline{\lambda}(-A_{s}) T_{s}.
\end{align*}
A fortiori,
\begin{align*}
    \Vert \tilde{\sigma}(t) \Vert
    \leq m_{f}^{-1} \sqrt{n_{v}}
    (&\chi_{1}(t) + \chi_{2}(t) + \chi_{3}(t))
    \\
    \leq m_{f}^{-1} \sqrt{n_{v}} \Big(&
        (2  \Delta_{\dot{\sigma}} + \underline{\lambda}(-A_{s}) \Delta_{\sigma} )
 \Delta_{\nabla_{vv} \Phi}
 \\
        &+ \Delta_{\sigma} \Delta_{\dot{\nabla}_{vv} \Phi}
    \Big) T_{s}
    =: \zeta_{2}(A_{s}) T_{s},
\end{align*}
for all $t \in [T_{s}, \tau]$, which concludes the proof.

\end{proof}

\subsection{Proof of \autoref{thm:l1ao}}
\label{sec:proof_of_tube}
The proof is done by contradiction.
Suppose $v(t) \in \mathcal{O}_{\nabla}(t, \rho), \forall t \geq 0$ is not true.
From $v(0) \in \mathcal{O}_{\nabla}(0, \rho)$ and the continuity of $\nabla_{v} \Phi$,
there exists $t^{\star} \in (0, \infty)$ such that $\Vert \nabla_{v} \Phi(t^{*}, v(t^{*})) \Vert = \rho$ and $\Vert \nabla_{v} \Phi(t, v(t)) \Vert < \rho, \forall t \in [0, t^{*})$.
Let
\begin{equation}
\begin{split}
    \eta(t)
    &:= \mathcal{L}^{-1}[C(s) \sigma(s)],
    \\
    \hat{\eta}(t)
    &:= \mathcal{L}^{-1}[C(s) \hat{\sigma}(s)] = -\dot{v}_{a}(t),
    \\
    \nu^{\intercal}(t)
    &:= \frac{\partial V(\nabla_{v} \Phi(t, v(t)))}{\partial \nabla_{v} \Phi} \nabla_{vv} \Phi(t, v(t)),
    \\
    \Delta_{\nu}
    &:= \Delta_{\nabla_{\frac{\partial V}{\partial \nabla_{v} \Phi}}} \Delta_{\nabla_{vv} \Phi},
\end{split}
\end{equation}
where $\mathcal{L}^{-1}$ denotes the inverse Laplace transformation.
Then, we have
\begin{equation}
\begin{aligned}
    \dot{V}
    &= \frac{\partial V}{\partial \nabla_{v} \Phi} \dot{\nabla}_{v} \Phi
    = \frac{\partial V}{\partial \nabla_{v} \Phi} (
        \hat{\nabla}_{vt} \Phi
        + \nabla_{vv} \Phi (\dot{v}_{b} + \dot{v}_{a} + \sigma)
    )
    \\
    &\leq -2\beta V
    + \frac{\partial V}{\partial \nabla_{v} \Phi}
     \nabla_{vv} \Phi ( \sigma(t) + \dot{v}_{a}(t))   
    \\
    &= -2\beta V +\nu^{\intercal}(t) (\sigma(t) - \eta(t)) + \nu^{\intercal}(t)( \eta(t) - \hat{\eta}(t)).
\end{aligned}
\end{equation}
By comparison lemma~\cite{khalilNonlinearSystems2002}, we have
\begin{equation}
\label{eq:Lyapunov_ineq}
    V(t) \leq e^{-2\beta t} V_{0} + \chi_{4}(t) + \chi_{5}(t),
\end{equation}
where
\begin{equation}
\begin{split}
    \chi_{4}(t)
    &:= \int_{0}^{t} e^{-2\beta (t-\tau)} \nu^{\intercal}(\tau) (\sigma(\tau) - \eta(\tau)) d\tau,
    \\
    \chi_{5}(t)
    &:= \int_{0}^{t} e^{-2\beta (t-\tau)} \nu^{\intercal}(\tau) (\eta(\tau) - \hat{\eta}(\tau)) d\tau.
\end{split}
\end{equation}

From \autoref{assumption:tube} and \autoref{lemma:uncertainty_est_bound},
$\Vert \nu \Vert_{\mathcal{L}^{\infty}}^{[0, t^{*}]}
\leq \Delta_{\nu}$
and $\Vert \dot{\sigma} \Vert_{\mathcal{L}^{\infty}}^{[0, t^{*}]} \leq \Delta_{\dot{\sigma}}$.
Since $\chi_{4}$ is the solution to the following equation
\begin{align}
    \dot{\chi}_{4}(t)
    &= -2\beta \chi_{4}(t) + \nu^{\intercal}(t) \phi(t), \quad \chi_{4}(0) = 0,
    \\
    \phi(s)
    &= (I_{n_{v}} - C(s)) \sigma(s),
\end{align}
and $\Vert \sigma(0) \Vert \leq \Delta_{\sigma}$,
it can be shown by \cite[Lemma A.1]{lakshmananSafeFeedbackMotion2020a} that
\begin{equation}
\label{eq:chi_4_bound}
    \Vert \chi_{4} \Vert_{\mathcal{L}^{\infty}}^{[0, t^{*}]}
    \leq
    \Delta_{\nu}
    \left(
        \frac{\Delta_{\sigma}}{\vert 2\beta - \omega \vert} + \frac{\Delta_{\dot{\sigma}}}{2\beta \omega}
    \right)
    =: \zeta_{1} (\omega).
\end{equation}

Let $\tilde{\eta}(t) := \hat{\eta}(t) - \eta(t)$.
The dynamics of $\tilde{\eta}$ are given as
$\dot{\tilde{\eta}}(t) = -\omega \tilde{\eta}(t) - \omega\sigma(t), \forall t \in [0, T_{s})$,
and the solution is $\tilde{\eta}(t) = -\int_{0}^{t} e^{-\omega(t- \tau)} \omega \sigma(\tau) d \tau$.
Therefore,
for all $t \in [0, T_{s})$, we have
\begin{align*}
    &\Vert \tilde{\eta}(t) \Vert
    \leq \Delta_{\sigma} \int_{0}^{t} \omega e^{-\omega(t-\tau)} d \tau
    = \Delta_{\sigma} (1-e^{-\omega t})
    \leq \Delta_{\sigma} \omega t
    \\
    \stepcounter{equation}\tag{\theequation}
    &\leq \Delta_{\sigma} \omega T_{s}
    =: \zeta_{3}(\omega) T_{s}
    < (\zeta_{2}(A_{s}) + \zeta_{3}(\omega)) T_{s}.
\end{align*}
For all $t \geq T_{s}$,
the dynamics are given by  $\dot{\tilde{\eta}}(t) = -\omega \tilde{\eta}(t) + \omega\tilde{\sigma}(t)$,
and the solution is $\tilde{\eta}(t) = e^{-\omega (t-T_{s})} \tilde{\eta}(T_{s})
+ \int_{T_{s}}^{t} e^{-\omega(t- \tau)} \omega \tilde{\sigma}(\tau) d \tau$.
Then, we have
\begin{align*}
    \Vert \tilde{\eta}(t) \Vert
    &\leq e^{-\omega(t-T_{s})} \Vert \tilde{\eta} (T_{s}) \Vert
    + \int_{T_{s}}^{t} e^{-\omega(t-\tau)} \omega \Vert \tilde{\sigma}(\tau) \Vert d\tau
    \\
    &\leq (\zeta_{3}(\omega) +  (1 - e^{-\omega (t-T_{s})}) \zeta_{2} (A_{s})) T_{s}
    \\
    &\leq (\zeta_{2} (A_{s}) + \zeta_{3}(\omega)) T_{s}.
    \stepcounter{equation}\tag{\theequation}
\end{align*}
Therefore, $\Vert \tilde{\eta} \Vert_{\mathcal{L}^{\infty}}^{[0, t^{*}]}
\leq (\zeta_{2}(A_{s}) + \zeta_{3}(\omega)) T_{s}$,
which implies
\begin{equation}
\label{eq:chi_5_bound}
\begin{aligned}
    \Vert \chi_{5}(t) \Vert
    &\leq 
    \Delta_{\nu}
    \int_{0}^{t} e^{-2\beta(t-\tau)} \Vert \tilde{\eta}(t) \Vert d\tau
    \\
    &\leq
    \Delta_{\nu}
    \frac{\zeta_{2}(A_{s}) + \zeta_{3}(\omega)}{2\beta} T_{s}
    =: \zeta_{4}(A_{s}, \omega) T_{s},
\end{aligned}
\end{equation}
for all $t \in [0, t^{*}]$.

From \eqref{eq:Lyapunov_ineq}, \eqref{eq:chi_4_bound}, and \eqref{eq:chi_5_bound},
we have
\begin{equation}
\label{eq:Lyapunov_ineq2}
\begin{split}
    V(t)
    &\leq e^{-2\beta t} V_{0} + \zeta_{1}(\omega) + \zeta_{4}(A_{s}, \omega) T_{s},
    \\
    &< V_{0} + \zeta_{1}(\omega) + \zeta_{4}(A_{s}, \omega) T_{s},
\end{split}
\end{equation}
for all $t \in (0, t^{*}]$, which implies
\begin{equation}
    V(t)
    < V_{0} + \zeta_{1}(\omega) + \zeta_{4} (A_{s}, \omega) T_{s},
\end{equation}
for all $t \in [0, t^{*}]$.
Also, $\Vert \nabla_{v} \Phi(t^{*}, v(t^{*})) \Vert = \rho$.
From \autoref{assumption:baseline_optimizer},
\begin{equation}
\begin{split}
 &\underline{\alpha} \rho^{2}
 \leq V(t^{*})
 < V_{0} + \zeta_{1}(\omega) + \zeta_{4}(A_{s}, \omega) T_{s}
\\
 &\Rightarrow
T_{s}
> \frac{\underline{\alpha} \rho^{2} - V_{0} - \zeta_{1}(\omega)}{\zeta_{4}(A_{s}, \omega)},
\end{split}
\end{equation}
which contradicts \eqref{eq:conditions_sampling_time_and_filter_bandwidth}.
Therefore, $v(t) \in \mathcal{O}_{\nabla}(t, \rho), \forall t \geq 0$.
The first part of \eqref{eq:UUB} can be deduced from \autoref{assumption:baseline_optimizer} and \eqref{eq:Lyapunov_ineq2}.

The rest of the proof is straightforward from \autoref{lemma:ball_tube}.

\subsection{Proof of \autoref{corollary:cost_bound}}
\label{sec:proof_of_cost_tube}
The left-hand side of the inequality is trivial.
For the right-hand side,
the convexity of $\Phi(t, \cdot)$ implies
\begin{equation}
    \Phi(t, v(t)) - \Phi(t, v^{\star}(t))
    \leq \nabla_{v} \Phi(t, v(t))^{\intercal}(v(t) - v^{\star}(t)),
\end{equation}
and $\nabla_{v} \Phi(t, v(t))^{\intercal}(v(t) - v^{\star}(t))
\leq \Vert \nabla_{v} \Phi(t, v(t)) \Vert \Vert v(t) - v^{\star}(t) \Vert$,
which concludes the proof by \autoref{thm:l1ao}.

\subsection{Simulation settings}
\label{sec:simulation_settings}
All simulations are performed on an M1 Macbook Air.
Fixed-time Euler integration is used for numerical simulation with the time step $\Delta t = 10^{-3}$.

\subsubsection{Example 1}
For the problem in \autoref{sec:example1},
we compare the following methods:
i) modified PCIP in \eqref{eq:ip_tvopt_modified} with $\epsilon=1$
and ii) $\mathcal{L}_{1}$-AO in \eqref{eq:l1ao_alg} with the PCIP method in \eqref{eq:ip_tvopt} as the baseline optimizer.
The optimal solution is obtained using a time-invariant solver, ECOS~\cite{domahidiECOSSOCPSolver2013}.
The gain of the PCIP method is $P=10$ by default,
and the parameters of the $\mathcal{L}_{1}$-AO are set as
$T_{s} = 10^{-3}$,
$\omega = 10^{3}$,
$A_{s} = - 1$.

\subsubsection{Example 2}
For the problem in \autoref{sec:example2},
we compare the following methods: i) PCIP vs PCIP+L1AO (i.e., $\mathcal{L}_{1}$-AO with the baseline of PCIP), ii) Modified PCIP vs Modified PCIP+L1AO (i.e., $\mathcal{L}_{1}$-AO with the baseline of modified PCIP).
The optimal solution is obtained using a time-invariant solver, ECOS~\cite{domahidiECOSSOCPSolver2013}.
The gain of PCIP is set as $P=I_{2}$,
and the gain of the modified PCIP is set as $P=10 I_{2}$ with $\epsilon = 0.1$.
The parameters of $\mathcal{L}_{1}$-AO are set as
$T_{s} = 10^{-3}$, $\omega=50$, and $A_{s} = -0.1I_{2}$.
The simulation time is set as $t_{f} = 50$.
The target position is modeled as $x_{d}(t) = 15 \left[\cos\left( \frac{2\pi t}{t_{f}}\right), \sin\left( \frac{2\pi t}{t_{f}}\right)\right]^{\intercal}$.
The penalty parameter for interior-point formulation is set as $c(t) = 50 e^{t/50}$.

\fi

\end{document}